\documentclass{article}
\usepackage{agd-paper}
\usepackage[margin=1.0in]{geometry}

\author{Andrew Gainer-Dewar \\ Carleton College \\ \url{againerdewar@carleton.edu}}
\title{Enumeration of labeled and unlabeled $k$-gonal and polygonal $2$-trees and succulents by vertices}

\usepackage{hyperref}

\newcommand*{\gcname}[1]{\ensuremath{\mathfrak{#1}}}
\newcommand*{\specname}[1]{\ensuremath{\mathcal{#1}}}
\newcommand*{\symgp}[1]{\ensuremath{\mathfrak{S}_{#1}}}

\newcommand*{\specsucc}{\ensuremath{\gcname{s}}}
\newcommand*{\specpt}{\ensuremath{\gcname{a}_{\mathrm{p}}}}
\newcommand*{\specpto}[1][]{\ensuremath{\gcname{a}_{\mathrm{p,O}}^{#1}}}
\newcommand*{\speckt}[1]{\ensuremath{\gcname{a}_{2, {#1}}}}
\newcommand*{\speckto}[2][]{\ensuremath{\gcname{a}_{2,{#2},O}^{#1}}}

\newcommand*{\specset}{\ensuremath{\specname{E}}}
\newcommand*{\specsing}[1]{\ensuremath{#1}}
\newcommand*{\speclin}[1][]{\ensuremath{\specname{L}_{#1}}}
\newcommand*{\speccyc}[1][]{\ensuremath{\specname{C}_{#1}}}

\newcommand*{\cycgraph}[1]{\ensuremath{C_{#1}}}

\newcommand*{\pointed}[2][]{\ensuremath{{#2}^{\bullet {#1}}}}

\newcommand*{\ci}[2][]{ \ensuremath{ Z_{#2}^{#1} } }
\newcommand*{\gci}[3][]{ \ci[{#2} {#1}]{#3} }

\DeclarePairedDelimiter{\pbrac}{(}{)}

\begin{document}
\maketitle

\begin{abstract}
  We use the theory of $\Gamma$-species to enumerate $k$-gonal and polygonal $2$-trees with respect to their vertices.
  We then extend this result to enumerate ``succulents'', a tree-like class of graphs which generalize cacti.
\end{abstract}

\section{Introduction}
\label{sec:intro}
A \emph{$2$-tree} is a connected simple graph obtained by beginning with an edge graph and then iteratively adding a vertex and connecting it by edges to the endpoints of an existing edge.
The result is, effectively, a tree-like assembly of triangles ``glued together'' along their edges.
It is natural to extend to the case where these components may be polygons of more than three sides.
The result is a \emph{$k$-gonal $2$-tree} if all of the pieces are $k$-gons and a \emph{polygonal $2$-tree} if the pieces may be polygons of any number of sides.
(The classical $2$-trees may be recovered from this definition as $3$-gonal $2$-trees.)
See \cref{fig:ptree} for an example of a generic polygonal $2$-tree.

\begin{figure}[htb]
  \centering
  \begin{tikzpicture}[node distance=1.5cm]
    \node[style=cnode] (a) at (0, 0) {$a$};
    \node[style=cnode] (h) [right of=a] {$h$};
    \node[style=cnode] (d) [above left of=a] {$d$};
    \node[style=cnode] (i) [below left of=d] {$i$};
    \node[style=cnode] (b) [above left of=i] {$b$};
    \node[style=cnode] (g) [above left of=d] {$g$};
    \node[style=cnode] (c) [above right of=d] {$c$};
    \node[style=cnode] (f) [right of=c] {$f$};
    \node[style=cnode] (e) [above of=f] {$e$};

    \path [style=edge]
    (a) edge (c)
    (c) edge (f)
    (f) edge (h)
    (h) edge (a)
    (f) edge (e)
    (e) edge (c)
    (c) edge (d)
    (d) edge (a)
    (c) edge (g)
    (g) edge (b)
    (b) edge (i)
    (i) edge (a)
    ;
  \end{tikzpicture}
  \caption{A (vertex-labeled) polygonal $2$-tree with cycles of lengths $3$, $4$, and $5$}
  \label{fig:ptree}
\end{figure}
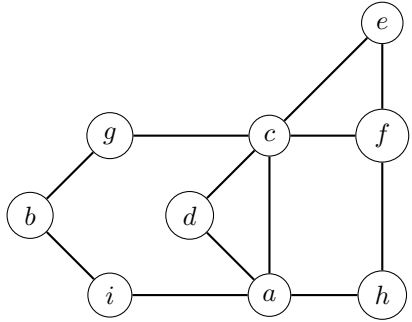

There is an extensive enumerative literature on $2$-trees.
They were first enumerated in the labeled case by Palmer \cite{palmer} and in the unlabeled case four years later by Harary and Palmer \cite{hpge}.
$k$-gonal $2$-trees were first considered by Harary, Palmer, and Read in \cite{hprk2t}, in which they consider only the case that at most two polygons may share an edge (yielding graphs which are `outerplanar').
This work was sharpened by Ducharme, Labelle, Lamanthe, and Leroux in \cite{dlll}, using the theory of combinatorial species.
Labelle, Lamanthe, and Leroux also enumerated labeled and unlabeled $k$-gonal $2$-trees (without the outerplanarity restriction) in \cite{lll}, extending previous work on (classical or $3$-gonal) $2$-trees by Fowler, Gessel, Labelle, and Leroux in \cite{spec2trees}.
General polygonal $2$-trees have received less attention.

In \cref{sec:pt}, we consider polygonal $2$-trees from a perspective informed by the theory of $\Gamma$-species, introduced by Henderson in \cite{hend} and developed in the author's previous work \cite{argthesis,agdktrees} as an enumerative tool in order to count $k$-trees, which generalize $2$-trees by gluing $\pbrac{k+1}$-simplices along their $k$-sub-simplices.
A notion of orientation is introduced, and several rooted variants of these `coherently-oriented' polygonal $2$-trees are characterized by a system of functional relationships in \cref{eq:ptorootestar,eq:ptoroote,eq:ptorootp,eq:ptorootpe}.
The orientations are removed by taking a quotient under a group action, and the roots are removed through a `dissymmetry theorem' expressed in \cref{eq:dissymalt}.
This allows computation of the numbers of labeled and unlabeled polygonal $2$-trees with $n$ vertices, shown in \cref{tab:ptunrooted}, with code given in \cref{sec:code}.

The method used is thematically similar to that employed in \cite{lll}, but offers two advantages.
First, by employing the more powerful algebraic and structural tools associated with the theory of $\Gamma$-species, it avoids entanglement in subtleties which would make application of the earlier methods to polygonal $2$-trees very complicated.
Second, it allows enumeration with respect to number of vertices, while the earlier work enumerates with respect to number of edges.
This distinction is not important in the case of $k$-gonal $2$-trees, because the edge and vertex counts are coupled.
However, in general polygonal $2$-trees, this is not the case; for example, for a fixed number $n \geq 4$ of vertices, the $n$-gon is a polygonal $2$-tree with $n$ vertices and $n$ edges, while a classical $2$-tree with $n$ vertices has $2n-3$ edges.

In \cref{sec:kt}, we show that small modifications to the work in \cref{sec:pt} allows for enumeration of $k$-gonal $2$-trees for fixed $k$.

Finally, \cref{sec:succ}, we apply the results of \cref{sec:pt} to enumerate the \emph{succulents} defined by Wilkes in \cite{vcuts}, by showing that the succulents are exactly the connected graphs whose blocks are polygonal $2$-trees and then applying the theory of block decompositions.

See \cref{fig:succ} for an example of a succulent graph.

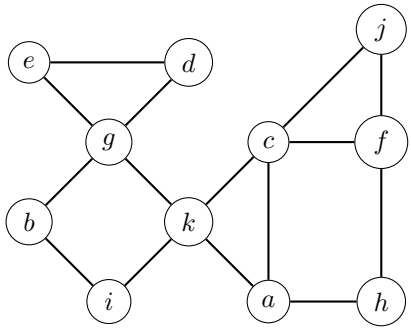
\begin{figure}[htb]
  \centering
  \begin{tikzpicture}[node distance=1.5cm]
    \node[style=cnode] (a) at (0, 0) {$a$};
    \node[style=cnode] (h) [right of=a] {$h$};
    \node[style=cnode] (k) [above left of=a] {$k$};
    \node[style=cnode] (c) [above right of=k] {$c$};
    \node[style=cnode] (f) [right of=c] {$f$};
    \node[style=cnode] (j) [above of=f] {$j$};
    \node[style=cnode] (i) [below left of=d] {$i$};
    \node[style=cnode] (g) [above left of=d] {$g$};
    \node[style=cnode] (b) [below left of=g] {$b$};
    \node[style=cnode] (e) [above left of=g] {$e$};
    \node[style=cnode] (d) [above right of=g] {$d$};

    \path [style=edge] (a) -- (h) -- (f) -- (j) -- (c) -- (f)
    (c) -- (a) -- (k) -- (c)
    (k) -- (g) -- (e) -- (d) -- (g)
    (g) -- (b) -- (i) -- (k);
  \end{tikzpicture}
  \caption{A (vertex-labeled) succulent}
  \label{fig:succ}
\end{figure}

\section{Polygonal $2$-trees}
\label{sec:pt}
We first investigate the species $\specpt$ of polygonal $2$-trees.
We adopt the following definition:
\begin{definition}
  Any cycle graph $\cycgraph{k}$ is a polygonal $2$-tree, and any graph obtained from a polygonal $2$-tree $T$ by adding a path of any length $k \geq 1$, selecting an edge $e$ of $T$, and connecting the end vertices of the path to the two ends of $e$ by edges so as to induce a cycle of length $k+2$ is a polygonal $2$-tree.
\end{definition}

Effectively, polygonal $2$-trees are ``trees made of polygons''; see \cref{fig:ptree} for an illustration.

The tree-like structure of polygonal $2$-trees suggests that we may find some recursive functional equations for their cycle index series.
Before we can do this, however, we must add additional structure.

\begin{definition}
  \label{def:orientation}
  An \emph{orientation} of a polygonal $2$-tree is an assignment of a direction to each of its edges.
  An orientation of a polygonal $2$-tree is \emph{coherent} if there is a directed path around each of its cycles.
\end{definition}

See \cref{fig:coptree} for an illustration of a coherent orientation of the polygonal $2$-tree from \cref{fig:ptree}.

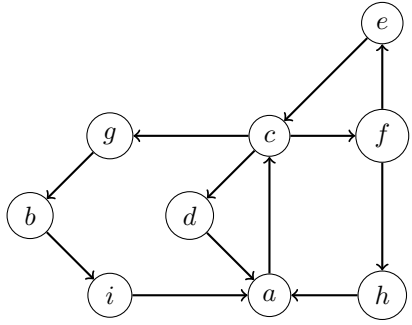
\begin{figure}[htb]
  \centering
  \begin{tikzpicture}[node distance=1.5cm]
    \node[style=cnode] (a) at (0, 0) {$a$};
    \node[style=cnode] (h) [right of=a] {$h$};
    \node[style=cnode] (d) [above left of=a] {$d$};
    \node[style=cnode] (i) [below left of=d] {$i$};
    \node[style=cnode] (b) [above left of=i] {$b$};
    \node[style=cnode] (g) [above left of=d] {$g$};
    \node[style=cnode] (c) [above right of=d] {$c$};
    \node[style=cnode] (f) [right of=c] {$f$};
    \node[style=cnode] (e) [above of=f] {$e$};

    \path [style=diredge]
    (a) edge (c)
    (c) edge (f)
    (f) edge (h)
    (h) edge (a)
    (f) edge (e)
    (e) edge (c)
    (c) edge (d)
    (d) edge (a)
    (c) edge (g)
    (g) edge (b)
    (b) edge (i)
    (i) edge (a)
    ;
  \end{tikzpicture}
  \caption{A (vertex-labeled) coherently-oriented polygonal $2$-tree}
  \label{fig:coptree}
\end{figure}

We note that any polygonal $2$-tree admits exactly two coherent orientations.
The class $\specpto$ of coherently-oriented polygonal $2$-trees then admits a natural action of the two-element group $\symgp{2}$ in which the trivial element $e$ acts trivially and the non-trivial element $\tau$ acts by reversing the direction of every edge.
This action commutes with permutations of labels, so $\specpto$ is a $\symgp{2}$-species with respect to this action in the sense of \cite{hend}.

\begin{theorem}
  \label{thm:ptquotient}
  The species $\specpt$ of unrooted, unoriented polygonal $2$-trees and the $\symgp{2}$-species $\specpto$ of unrooted, coherently-oriented polygonal $2$-trees satisfy the species isomorphism
  \begin{equation}
    \label{eq:ptquotient}
    \specpt = \quotient{\specpto}{\symgp{2}}
  \end{equation}
  in the sense of \cite{argthesis} (indicating that each polygonal $2$-tree may naturally be identified with an orbit of coherently-oriented polygonal $2$-trees under the action of $\symgp{2}$).
\end{theorem}

Thus, if we can compute the $\symgp{2}$-cycle index $\gci{\symgp{2}}{\specpto}$ for coherently-oriented polygonal $2$-trees, we may pass to the ordinary cycle index $\ci{\specpt}$ for polygonal $2$-trees.

\subsection{Rooted coherently-oriented polygonal $2$-trees}
\label{sec:ptoroot}
Let $\specpto[e \star]$ denote the $\symgp{2}$-species of coherently-oriented polygonal $2$-trees rooted at an unlabeled edge whose end vertices are unlabeled, $\specpto[e]$ denote the $\symgp{2}$-species of coherently-oriented polygonal $2$-trees rooted at an edge, $\specpto[p]$ denote the $\symgp{2}$-species of coherently-oriented polygonal $2$-trees rooted at a polygon, and $\specpto[pe]$ denote the $\symgp{2}$-species of coherently-oriented polygonal $2$-trees rooted at a polygon with a distinguished edge.
(We consider these to be $3$-sort $\symgp{2}$-species, where vertices are of sort $\specsing{X}$, edges are of sort $\specsing{Y}$, and polygons are of sort $\specsing{Z}$.)

Further, let $\speclin$ denote the $\symgp{2}$-species of linear orders and $\speccyc$ the $\symgp{2}$-species of cyclic orders, in each case with the $\symgp{2}$-action whose nontrivial element reverses order.
(In each case, we will use a numeric subscript `$n$' (e.g.~$\speclin[2]$) to denote the restriction to terms of cardinality $n$ and an inequality subscript `$\geq n$' (e.g.~$\speccyc[\geq 4]$) to denote the restriction to terms of cardinality no less than $n$.)
Additionally, let $\specset$ denote the $\symgp{2}$-species of sets with the trivial $\symgp{2}$-action.

\begin{theorem}
  \label{thm:ptoroot}
  The $3$-sort $\symgp{2}$-species $\specpto[e \star]$, $\specpto[e]$, $\specpto[p]$, and $\specpto[pe]$ of variously-rooted coherently-oriented polygonal $2$-trees satisfy the functional equations
  \begin{subequations}
    \label{eq:ptoroot}
    \begin{align}
      \specpto[e \star] \pbrac{\specsing{X}, \specsing{Y}, \specsing{Z}} &= \specset \circ \pbrac[\bigg]{\specsing{Z} \cdot \sum_{n \geq 1} \pbrac*{\speclin[n] \pbrac{X} \cdot \speclin[n+1] \pbrac{\specsing{Y} \cdot \specpto[e \star]}}}
      \label{eq:ptorootestar} \\
      \specpto[e] \pbrac{\specsing{X}, \specsing{Y}, \specsing{Z}} &= \specsing{Y} \cdot \speclin[2] \pbrac{\specsing{X}} \cdot \pbrac*{\specpto[e \star] - 1}
      \label{eq:ptoroote} \\
      \specpto[p] \pbrac{\specsing{X}, \specsing{Y}, \specsing{Z}} &= \specsing{Z} \cdot \speccyc[\geq 3] \pbrac*{\specsing{X} \cdot \specsing{Y} \cdot \specpto[e \star]}
      \label{eq:ptorootp} \\
      \specpto[pe] \pbrac{\specsing{X}, \specsing{Y}, \specsing{Z}} &= \speclin[2] \pbrac*{\specsing{X}} \cdot \specsing{Y} \cdot \specsing{Z} \cdot \specpto[e \star] \cdot \sum_{n \geq 1} \pbrac*{\speclin[n] \pbrac{X} \cdot \speclin[n+1] \pbrac{\specsing{Y} \cdot \specpto[e \star]}}
      \label{eq:ptorootpe}
    \end{align}
  \end{subequations}
  as isomorphisms of $3$-sort $\symgp{2}$-species.
\end{theorem}

\begin{proof}
  We consider each of \cref{eq:ptorootestar,eq:ptoroote,eq:ptorootp,eq:ptorootpe} in turn, demonstrating equality by exhibiting an explicit isomorphism between the left- and right-hand sides.

  Consider an $\specpto[e \star]$-structure---that is, a coherently-oriented polygonal $2$-tree rooted at an unlabeled edge whose end vertices are unlabeled.
  (See \cref{fig:ptorootestar} for an example.)
  Attached to the root edge are several ``sheets''\footnote{We adopt the term ``sheets'' as it is used in \cite{spec2trees}.}--- that is, coherently-oriented polygonal $2$-trees rooted at a polygon which is missing an edge.
  These sheets may be freely interchanged without affecting the $\specpto[e \star]$-structure and are invariant under the action of $\symgp{2}$, so they form an $\specset$-ensemble.
  Each sheet has a $\specsing{Z}$-label at its root polygon and a linearly-ordered sequence of descendant structures.
  Specifically, there are $n$ $\specsing{X}$-structures and $n+1$ $\specsing{Y}$-structures for some $n \geq 1$; these $\specsing{Y}$-structures are in turn the roots of descendant $\specpto[e \star]$-structures.
  No permutation of the $\specsing{X}$- or $\specsing{Y}$-structures fixes the overall sheet structure, and the action of the nontrivial element of $\symgp{2}$ reverses the overall order, so these descendants taken together form an $\pbrac*{\speclin[n] \pbrac{\specsing{X}} \cdot \speclin[n+1] \pbrac{\specsing{Y} \cdot \specpto[e \star]}}$-structure.
  Since $n$ may take any value $\geq 1$, \cref{eq:ptorootestar} follows.

  \begin{figure}[htb]
    \centering
    \begin{tikzpicture}[node distance=2.5cm]
      \node[style=cnode] (a) at (0, 0) {$\ast$};
      \node[style=cnode] (h) [right of=a] {$\specsing{X}$};
      \node[style=cnode] (d) [above left of=a] {$\specsing{X}$};
      \node[style=cnode] (i) [below left of=d] {$\specsing{X}$};
      \node[style=cnode] (b) [above left of=i] {$\specsing{X}$};
      \node[style=cnode] (g) [above left of=d] {$\specsing{X}$};
      \node[style=cnode] (c) [above right of=d] {$\ast$};
      \node[style=cnode] (f) [right of=c] {$\specsing{X}$};
      \node[style=cnode] (e) [above of=f] {$\specsing{X}$};

      \node[style=dnode] (A) at (barycentric cs:a=1/2,f=1/2) {$\specsing{Z}$};
      \node[style=dnode] (B) at (barycentric cs:c=1/4,f=1/2,e=1/4) {$\specsing{Z}$};
      \node[style=dnode] (C) at (barycentric cs:a=1/4,c=1/4,d=1/2) {$\specsing{Z}$};
      \node[style=dnode] (D) at (barycentric cs:d=3/5,b=2/5) {$\specsing{Z}$};

      \draw [style={dirrootedge}]
      (a) -- node [midway,style=snode] {$\ast$} (c);
      
      \path [style=diredge]
      (c) edge node [midway,style=snode] {$\specsing{Y}$} (f)
      (f) edge node [midway,style=snode] {$\specsing{Y}$} (h)
      (h) edge node [midway,style=snode] {$\specsing{Y}$} (a)
      (f) edge node [midway,style=snode] {$\specsing{Y}$} (e)
      (e) edge node [midway,style=snode] {$\specsing{Y}$} (c)
      (c) edge node [midway,style=snode] {$\specsing{Y}$} (d)
      (d) edge node [midway,style=snode] {$\specsing{Y}$} (a)
      (c) edge node [midway,style=snode] {$\specsing{Y}$} (g)
      (g) edge node [midway,style=snode] {$\specsing{Y}$} (b)
      (b) edge node [midway,style=snode] {$\specsing{Y}$} (i)
      (i) edge node [midway,style=snode] {$\specsing{Y}$} (a)
      ;
    \end{tikzpicture}
    \caption{An example $\specpto[e \star]$-structure, with the sorts indicated by $\specsing{X}$, $\specsing{Y}$, and $\specsing{Z}$ and unlabeled root components indicated by $\ast$}
    \label{fig:ptorootestar}
  \end{figure}
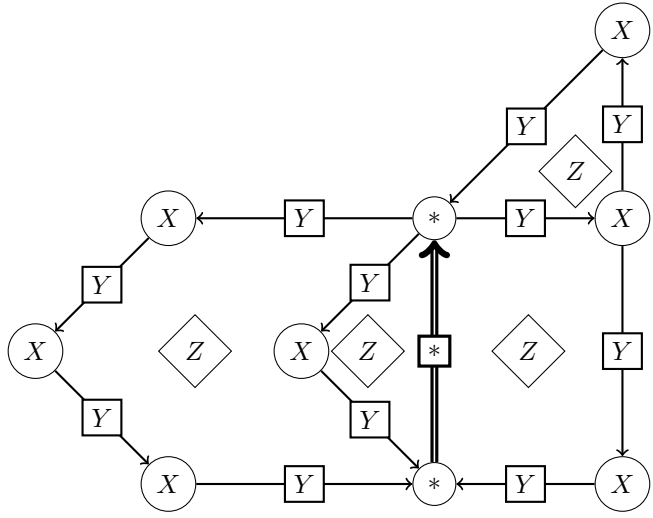

  Now consider an $\specpto[e]$-structure---that is, a coherently-oriented polygonal $2$-tree rooted at an edge.
  The edge has a $\specsing{Y}$-label, and each endpoint of that edge has an $\specsing{X}$-label; under the action of the nontrivial element of $\symgp{2}$, these two $\specsing{X}$-labels are interchanged, so they form an $\speclin[2] \pbrac{\specsing{X}}$-structure.
  What remains is a nonempty $\specpto[e \star]$-structure (since a single edge is not a polygonal $2$-tree).
  \Cref{eq:ptoroote} follows.

  Now consider an $\specpto[p]$-structure---that is, a coherently-oriented polygonal $2$-tree rooted at a polygon.
  The root polygon has a $\specsing{Z}$-label, and it is surrounded by at least three edges, each of which has an $\specsing{X}$-label on its target end, a $\specsing{Y}$-label on its edge, and a descendant $\specpto[e \star]$-structure.
  These descendants are cyclically ordered, and the action of the nontrivial element of $\symgp{2}$ reverses this ordering.
  \Cref{eq:ptorootp} follows.

  Finally, consider a $\specpto[pe]$-structure---that is, a coherently-oriented polygonal $2$-tree rooted at a polygon with a distinguished edge.
  Here, too, the root polygon has a $\specsing{Z}$-label, which is again surrounded by edges.
  Additionally, the distinguished edge has a $\specsing{Y}$-label, its endpoints have two $\specsing{X}$-labels which are interchanged when orientation is reversed, and the other polygons which may be attached there form a $\specpto[e \star]$-structure, contributing the term $\speclin[2] \pbrac*{\specsing{X}} \cdot \specsing{Y} \cdot \specpto[e \star]$.
  Finally, as in the case of $\specpto[e \star]$, the edges around the distinguished polygon together carry an $\pbrac*{\speclin[n] \pbrac{\specsing{X}} \cdot \speclin[n+1] \pbrac{\specsing{Y} \cdot \specpto[e \star]}}$-structure for some $n \geq 1$.
  \Cref{eq:ptorootpe} follows.
\end{proof}

\Cref{eq:ptorootestar,eq:ptoroote,eq:ptorootp,eq:ptorootpe} correspond to a recursive system of functional equations for the $\symgp{2}$-cycle indices $\gci{\symgp{2}}{\specpto[e]}$, $\gci{\symgp{2}}{\specpto[p]}$, and $\gci{\symgp{2}}{\specpto[pe]}$, which can be solved using algebraic and computational methods.

\subsection{Dissymmetry and generic coherently-oriented polygonal $2$-trees}
The results of \cref{sec:ptoroot} allow for the enumeration of coherently-oriented polygonal $2$-trees rooted at edges, polygons, or polygons with distinguished edges.
To pass to the unrooted case, we will use a `dissymmetry theorem', which connects these four species together.

We first define some preliminary notions.
The \emph{degree} of a given polygon $P$ in a polygonal $2$-tree $T$ is the number of edges of $p$ which are also edges of other polygons in $T$.
A polygon of degree $1$ is a \emph{leaf polygon}.
Note that any polygonal $2$-tree with at least two polygons has a leaf polygon.

Given a (coherenly-oriented or unoriented) polygonal $2$-tree $T$, its \emph{pruning} is the polygonal $2$-tree $T'$ that results when every leaf polygon of $T$ is removed (but the edges which attached them are left).
By the previous, if $T$ has at least two polygons, then its pruning has strictly fewer polygons, so iteratively pruning any polygonal $2$-tree will ultimately yield either a single edge or a single polygon.
This edge or polygon is the \emph{center} of $T$.
Note that the center is \emph{canonical} (i.e.~invariant under isomorphisms) and is independent of orientation.

\begin{theorem}[Dissymmetry theorem for coherently-oriented polygonal $2$-trees]
  \label{thm:dissym}
  The $\symgp{2}$-species $\specpto$, $\specpto[e]$, $\specpto[p]$, and $\specpto[pe]$ of unrooted and variously-rooted coherently-oriented polygonal $2$-trees satisfy the species isomorphism
  \begin{equation}
    \label{eq:dissym}
    \specpto + \specpto[pe] = \specpto[e] + \specpto[p].
  \end{equation}
\end{theorem}

\begin{proof}
  We will exhibit an explicit map from the right-hand side to the left-hand side.
  Let $T$ be a structure of the right-hand species $\specpto[e] + \specpto[p]$.
  Then $T$ is a coherently-oriented polygonal $2$-tree rooted at either an edge or a polygon.
  That root may be the center of $T$; in this case, we send $T$ to the $\specpto$-structure obtained by forgetting the root.
  
  Otherwise, consider the sub-tree $T'$ obtained by repeatedly pruning $T$ until its root is a leaf polygon (in the case that the root is a polygon) or the root is an edge of only one polygon (if the root is an edge).
  If the root is a polygon, let $R$ be the set of that root and the edge which attaches it to $T'$; if the root is an edge, let $R$ be the set of that edge and its polygon in $T'$.
  In either case, $R$ is a polygon with a distinguished edge.
  We then identify $T$ with the $\specpto[pe]$-structure which is the underlying tree of $T$ rooted at $R$.

  This map is clearly canonical and independent of orientation.
  Additionally, it is easily seen to be a bijection by constructing the inverse map (which sends each $\specpto$-structure to the same tree rooted at its center and sends each $\specpto[pe]$-structure to the same tree rooted at whichever component of the original root was ``away from'' the center).
  Thus, the map described is an isomorphism.
\end{proof}

\begin{corollary}
  \label{thm:dissymalt}
  The $\symgp{2}$-species $\specpto$ of unrooted coherently-oriented polygonal $2$-trees is characterized by the equation
  \begin{equation}
    \label{eq:dissymalt}
    \specpto = \specpto[e] + \specpto[p] - \specpto[pe].
  \end{equation}
\end{corollary}

For enumerative purposes, we are interested in \cref{eq:dissymalt} only at the $\symgp{2}$-cycle-index level, so we need not concern ourselves here with the details of the interpretation of the subtraction.
(The interested reader can find more details in \cite{argthesis}.)

\subsection{Unoriented polygonal $2$-trees}
Per \cref{eq:ptquotient}, the cycle index of the species $\specpt$ of polygonal $2$-trees is the quotient of the $\symgp{2}$-cycle index of the species $\specpto$ of coherently-oriented polygonal $2$-trees.
This completes the enumeration of polygonal $2$-trees.

The exact numbers of labeled and unlabeled polygonal $2$-trees are given in \cref{tab:ptunrooted}.

\section{$k$-gonal $2$-trees}
\label{sec:kt}
We now investigate the species $\speckt{k}$ of $k$-gonal $2$-trees (for $k \geq 3$).
We treat $k$-gonal $2$-trees as a special case of polygonal $2$-trees, so the definitions of \cref{sec:pt} may be adopted without modification.

We first present a version of \cref{thm:ptoroot} modified for the $k$-gonal case.
\begin{theorem}
  \label{thm:ktoroot}
  The $3$-sort $\symgp{2}$-species $\speckto[e \star]{k}$, $\speckto[e]{k}$, $\speckto[p]{k}$, and $\speckto[pe]{k}$ of variously-rooted coherently-oriented $k$-gonal $2$-trees satisfy the functional equations
  \begin{subequations}
    \label{eq:ptoroot}
    \begin{align}
      \speckto[e \star]{k} \pbrac{\specsing{X}, \specsing{Y}, \specsing{Z}} &= \specset \circ \pbrac[\bigg]{\specsing{Z} \cdot \speclin[k-2] \pbrac{X} \cdot \speclin[k-1] \pbrac*{\specsing{Y} \cdot \speckto[e \star]{k}}}
      \label{eq:ktorootestar} \\
      \speckto[e]{k} \pbrac{\specsing{X}, \specsing{Y}, \specsing{Z}} &= \specsing{Y} \cdot \speclin[2] \pbrac{\specsing{X}} \cdot \pbrac*{\speckto[e \star]{k} - 1}
      \label{eq:ktoroote} \\
      \speckto[p]{k} \pbrac{\specsing{X}, \specsing{Y}, \specsing{Z}} &= \specsing{Z} \cdot \speccyc[k] \pbrac*{\specsing{X} \cdot \specsing{Y} \cdot \speckto[e \star]{k}}
      \label{eq:ktorootp} \\
      \speckto[pe]{k} \pbrac{\specsing{X}, \specsing{Y}, \specsing{Z}} &= \speclin[2] \pbrac*{\specsing{X}} \cdot \specsing{Y} \cdot \specsing{Z} \cdot \speckto[e \star]{k} \cdot \sum_{n \geq 1} \pbrac*{\speclin[n] \pbrac{\specsing{X}} \cdot \speclin[n+1] \pbrac{\specsing{Y} \cdot \speckto[e \star]{k}}}
      \label{eq:ktorootpe}
    \end{align}
  \end{subequations}
  as isomorphisms of $3$-sort $\symgp{2}$-species.
\end{theorem}

\begin{proof}
  Proof is identical to that of \cref{thm:ptoroot}, with two modifications.
  First, in a $\speckto[e \star]{k}$-structure, each page attached to the root edge is a $k$-gon, so it carries $k - 2$ vertices and $k - 1$ edges, so the $\sum_{n \geq 1}$ term in \cref{eq:ptorootestar} is replaced with its $n = k-2$ term in constructing \cref{eq:ktorootestar}.
  Second, in a $\speckto[p]{k}$-structure or a $\specpto[pe]{k}$-structure, there are exactly $k$ edges on the root polygon, so the $\speccyc[\geq 3]$ term of \cref{eq:ptorootp} is replaced by a $\speccyc[k]$ term in \cref{eq:ktorootp} and the $\sum_{n \geq 1}$ term in \cref{eq:ptorootpe} is replaced with its $n = k-2$ term in \cref{eq:ktorootpe}.
\end{proof}

We next note that \cref{eq:ptquotient,thm:dissym} still apply:
\begin{corollary}
  \label{thm:dissymaltk}
  The $\symgp{2}$-species $\speckto{k}$ of unrooted coherently-oriented $k$-gonal $2$-trees is characterized by the equation
  \begin{equation}
    \label{eq:dissymaltk}
    \speckto{k} = \speckto[e]{k} + \speckto[p]{k} - \speckto[pe]{k}.
  \end{equation}
\end{corollary}

\begin{theorem}
  \label{thm:ktquot}
  The species $\speckt{k}$ of unrooted, unoriented $k$-gonal $2$-trees and the $\symgp{2}$-species $\speckto{k}$ of unrooted, coherently-oriented $k$-gonal $2$-trees are related by the species isomorphism
  \begin{equation}
    \label{eq:ktquotient}
    \speckt{k} = \quotient{\speckto{k}}{\symgp{2}}.
  \end{equation}
\end{theorem}

This completes the enumeration of $k$-gonal $2$-trees.

\section{Succulents}
\label{sec:succ}
We now investigate the species $\specsucc$ of succulent graphs.
We quote the definition of a succulent for convenience:
\begin{definition}[{\cite[Def.~8.1]{vcuts}}]
  \label{def:succulent}
  A \emph{succulent} is a connected graph built up from cycles (including possibly $2$-cycles, consisting of two vertices joined by a double edge) in the
  following manner.
  Two cycles may be joined together either at a single vertex or along a single edge.
\end{definition}

Recall that a \emph{block} of a graph $G$ is a maximal $2$-connect subgraph of $G$.
Any connected graph $G$ admits a canonical ``block decomposition'' into its block subgraphs, with the caveat that this ``decomposition'' is \emph{not} a partition of its vertex set; any \emph{cut vertex}\footnote{A \emph{cut vertex} of a connected graph is a vertex whose removal disconnects the graph.} of $G$ will necessarily lie in more than one block.
It is a classical result of structural graph theory that the block decomposition of a graph has a tree-like structure in the sense that there can be no ``loops of blocks''.
In light of the power of combinatorial species to study tree-like structures, this makes block decompositions a powerful way to enumerate classes of graphs.

In light of this, we can reformulate our description of a succulent graph, reframing \cref{def:succulent} in terms of blocks.
A succulent is assembled by starting with a cycle or edge graph and then iteratively attaching a cycle or edge graph to either a vertex or an edge of the existing graph.
This construction can create no `cycles of cycles', so any attachment at a single vertex induces a cut-vertex.
Thus, the two-connected succulents are exactly those which are constructed by using only edgewise attachments.
As a result, the blocks of succulents are exactly the polygonal $2$-trees of \cref{sec:pt}.

\begin{theorem}
  \label{thm:succblock}
  The succulents are exactly the connected graphs whose blocks\footnote{The \emph{blocks} of a graph are its maximal $2$-connected subgraphs.} are polygonal $2$-trees.
\end{theorem}

A vocabulary and toolset for studying classes of connected graphs and blocks is developed in \cite[\S 4.2]{bll}; in this language, the statement of \cref{thm:succblock} can be rephrased as the species equation
\begin{equation}
  \label{eq:succblock}
  \specsucc = C_{\specpt}.
\end{equation}

By Proposition 2 and Theorem 3 of \cite[\S 4.2]{bll}, we then have the following two additional relationships:
\begin{gather}
  \specsucc = \pointed{\specsucc} + \specpt \pbrac*{\pointed{\specsucc}} - \pointed{\specsucc} \cdot \specpt' \pbrac*{\pointed{\specsucc}} \label{eq:succblockdecomp} \\
  \intertext{where}
  \pointed{\specsucc} = X \cdot E \pbrac*{\specpt' \pbrac*{\pointed{\specsucc}}}. \label{eq:succpoint}
\end{gather}

Since the cycle index $\ci{\specpt}$ is known from \cref{sec:pt}, we can use \cref{eq:succpoint} to compute $\ci{\pointed{\specsucc}}$ recursively, then \cref{eq:succblockdecomp} to compute $\ci{\specsucc}$ from it.

The exact numbers of labeled and unlabeled succulent graphs are given in \cref{tab:succ}.

\appendix
\section{Enumerative tables}
\label{sec:enumtab}
The various functional relationships in \cref{eq:ptorootestar,eq:ptoroote,eq:ptorootp,eq:ptorootpe,eq:dissymalt,eq:succblockdecomp} correspond to recursive systems of equations in the corresponding cycle indices.
Computational techniques can be used to solve for the values of the coefficients.
We have done so here, using code in \cref{sec:code} executed in the Sage computer algebra system \cite{sage}.
We present here the numbers of labeled and unlabeled polygonal $2$-trees with $n \leq 26$ vertices and the numbers of labeled and unlabeled succulents with $n \leq 19$ vertices.

The calculations in \cref{tab:ptunrooted} took approximately 320s on a modern desktop.
The calculations in \cref{tab:succ} took approximately 49s on the same system (but were constrained by memory).

\begin{table}[htp]
  \centering
  \begin{tabular}{r r r}
    \toprule
    $n$ & Labeled & Unlabeled \\
    \midrule
    0 & 0 & 0 \\
    1 & 0 & 0 \\
    2 & 1 & 0 \\
    3 & 1 & 1 \\
    4 & 9 & 2 \\
    5 & 142 & 4 \\
    6 & 3255 & 12 \\
    7 & 98031 & 35 \\
    8 & 3656548 & 146 \\
    9 & 162577332 & 638 \\
    10 & 8389712565 & 3202 \\
    11 & 492731139565 & 16812 \\
    12 & 32442804010386 & 92896 \\
    13 & 2366514029082534 & 526772 \\
    14 & 189407564735080783 & 3059529 \\
    15 & 16501454669316415995 & 18074277 \\
    16 & 1554438720577536961560 & 108363677 \\
    17 & 157423599814757566519336 & 657666274 \\
    18 & 17055697585856128847006697 & 4034258315 \\
    19 & 1968364932798990980350721817 & 24978270864 \\
    20 & 241066057385127358326660352030 & 155936687183 \\
    21 & 31225184482248201727492659433530 & 980693145568 \\
    22 & 4264939764724371509073783537878211 & 6208610766918 \\
    23 & 612621843178318008183525963968742151 & 39541690252881 \\
    24 & 92318664159675081116148301725731288868 & 253208231528625 \\
    25 & 14562874254239454682491677079887534079900 & 1629504665609635 \\
    26 & 2399897780180354666071878804962398006738525 & 10534360792342723 \\
    \bottomrule
  \end{tabular}
  \caption{Number of polygonal $2$-trees with $n$ vertices}
  \label{tab:ptunrooted}
\end{table}

\begin{table}[htp]
  \centering
  \begin{tabular}{r r r}
    \toprule
    $n$ & Labeled & Unlabeled \\
    \midrule
    0 & 0 & 0 \\
    1 & 1 & 1 \\
    2 & 0 & 0 \\
    3 & 1 & 1 \\
    4 & 9 & 2 \\
    5 & 157 & 5 \\
    6 & 3795 & 15 \\
    7 & 119346 & 53 \\
    8 & 4621708 & 227 \\
    9 & 212726529 & 1066 \\
    10 & 11345387805 & 5523 \\
    11 & 687946890790 & 30142 \\
    12 & 46736272993806 & 172227 \\
    13 & 3515975765492235 & 1012974 \\
    14 & 290136704987785747 & 6104629 \\
    15 & 26055571620539221320 & 37471623 \\
    16 & 2529614021758754876520 & 233595886 \\
    17 & 263997116122623681660241 & 1475082907 \\
    18 & 29471762512579341908184345 & 9418713822 \\
    19 & 3504426532914198495232154142 & 60723473472 \\
    \bottomrule
  \end{tabular}
  \caption{Number of succulents with $n$ vertices}
  \label{tab:succ}
\end{table}

\clearpage

\section{Code listing}
\label{sec:code}
The enumerations in \cref{sec:enumtab} were completed using the Sage computer algebra system \cite{sage}.
The code used is given in \cref{list:succcode}.
In particular, the number of unlabeled succulents with up to $n$ vertices may be computed by copying the code into a Sage notebook, modifying the last line with the desired value of $n$, and executing.
The enumeration of other species, and labeled enumerations, may be obtained by modifying the last line appropriately.
(The code is also available in a form that can be executed online through the author's website.)

WARNING: This code depends on patches which have not yet been accepted to Sage main.
(\#14347, \#14543, \#14846)

\begin{lstlisting}[caption=Sage code to compute numbers of succulents, label=list:succcode, language=Python, texcl=true]
# Python's default recursion limit is 1000, which is much too low for us
sys.setrecursionlimit(100000)

# Set up environment
from sage.combinat.species.stream import _integers_from

S2 = SymmetricGroup(2)

from sage.combinat.species.generating_series import CycleIndexSeriesRing
CISR = CycleIndexSeriesRing(QQ)

from sage.combinat.species.group_cycle_index_series import GroupCycleIndexSeriesRing
GCISR = GroupCycleIndexSeriesRing(S2,QQ)
G = GCISR.basis()
e,t = G.keys()
ge = G[e]
gt = G[t]

# Define some utility species
E = species.SetSpecies().cycle_index_series()
X = species.SingletonSpecies().cycle_index_series()
C = species.CycleSpecies(min=2).cycle_index_series()
L = species.LinearOrderSpecies(min=1).cycle_index_series()
L2 = species.LinearOrderSpecies(size=2).cycle_index_series()

# Define some utility $\symgp{2}$-species
Es = GCISR(E)
Xs = GCISR(X)

# Define the $\symgp{2}$-species of linear orders
def Ltgen():
    p = SymmetricFunctions(QQ).power()
    yield 0
    yield 0
    for n in _integers_from(1):
        yield p([2]*n)
        yield p([2]*n+[1])

Ls = ge*L + gt*CISR(Ltgen())
Ln = lambda n: ge*Ls[e].restricted(min=n,max=n+1) + gt*Ls[t].restricted(min=n,max=n+1)
Lgn = lambda n: ge*Ls[e].restricted(min=n) + gt*Ls[t].restricted(min=n)

# Define the $\symgp{2}$-species of cyclic orders
def Ctgen():
    p = SymmetricFunctions(QQ).power()
    yield 0
    yield 0
    for n in _integers_from(1):
        yield 1/2*p([2]*n)+1/2*p([2]*(n-1)+[1]*2)
        yield p([2]*n+[1])
    
Cs = ge*C + gt*CISR(Ctgen())
Cn = lambda n: ge*Cs[e].restricted(min=n,max=n+1) + gt*Cs[t].restricted(min=n,max=n+1)
Cgn = lambda n: ge*Cs[e].restricted(min=n) + gt*Cs[t].restricted(min=n)

# Define the $\symgp{2}$-species $\specpto[e \star]$ per \cref{eq:ptorootestar}
estarrooted = GCISR(1)
ltx = GCISR(0)
estarrooted.define((Es)(ltx))
ltx[e].define(CISR.sum_generator(estarrooted[e]^(n+1)*X^n for n in _integers_from(1)))
ltx[t].define(CISR.sum_generator((estarrooted[t]^(n))(L2)*estarrooted[t]*(X^(n))(L2) + (estarrooted[t]^(n))(L2)*(X^(n-1))(L2)*X for n in _integers_from(1)))

# Define the other $\symgp{2}$-species of rooted coherently-oriented polygonal $2$-trees per \cref{eq:ptoroote,eq:ptorootp,eq:ptorootpe}
prooted = Cgn(3)(Xs*estarrooted)
erooted = Ln(2)*estarrooted
perooted = erooted*ltx

# Define the $\symgp{2}$-species of unrooted coherently-oriented polygonal $2$-trees per \cref{eq:dissymalt}
unrooted = erooted + prooted - perooted

# Define the species of unoriented polygonal $2$-trees per \cref{eq:ptquotient}
p2trees = unrooted.quotient()

# Define the species $\pointed{\specsucc}$ per \cref{eq:succpoint}
succpoint = CISR(0)
succpoint.define(X*E(p2trees.derivative()(succpoint)))

# Define the species $\specsucc$ per \cref{eq:succblockdecomp}
succulents = succpoint + p2trees(succpoint) - succpoint*p2trees.derivative()(succpoint)

# Compute the numbers of polygonal $2$-trees on $n$ vertices for $n \leq 20$
# (Note that this sequence is $0$-indexed, so we need $21$ terms!)
p2trees.isotype_generating_series().counts(21)
\end{lstlisting}

\bibliography{sources}

\begin{thebibliography}{10}

\bibitem{bll}
Fran\c{c}ois Bergeron, Gilbert Labelle, and Pierre Leroux.
\newblock {\em Combinatorial species and tree-like structures}.
\newblock Number~67 in Encyclopedia of Mathematics and its Applications.
  Cambridge University Press, Cambridge, 1998.

\bibitem{dlll}
Martin Ducharme, Gilbert Labelle, C{\'{e}}dric Lamanthe, and Pierre Leroux.
\newblock A classification of outerplanar $k$-gonal $2$-trees.
\newblock In {\em Formal Power Series and Algebraic Combinatorics}, 2007.

\bibitem{spec2trees}
Tom Fowler, Ira Gessel, Gilbert Labelle, and Pierre Leroux.
\newblock The specification of 2-trees.
\newblock {\em Advances in Applied Mathematics}, 28(2):145--168, 2002.

\bibitem{argthesis}
Andrew Gainer.
\newblock {\em {$\Gamma$}-species, quotients, and graph enumeration}.
\newblock PhD thesis, Brandeis University, 2012.

\bibitem{agdktrees}
Andrew Gainer-Dewar.
\newblock $\gamma$-species and the enumeration of $k$-trees.
\newblock {\em The Electronic Journal of Combinatorics}, 19(4):P45, 2012.

\bibitem{hpge}
Frank Harary and Edgar Palmer.
\newblock {\em Graphical Enumeration}.
\newblock Academic Press, New York, 1973.

\bibitem{hprk2t}
Frank Harary, Edgar Palmer, and Ronald Read.
\newblock On the cell-growth problem for arbitrary polygons.
\newblock {\em Discrete Mathematics}, 11(3):371--389, 1965.

\bibitem{hend}
Anthony Henderson.
\newblock Species over a finite field.
\newblock {\em Journal of Algebraic Combinatorics}, 21(2):147--161, 2005.

\bibitem{lll}
Gilbert Labelle, C{\'{e}}dric Lamanthe, and Pierre Leroux.
\newblock Labelled and unlabelled enumeration of $k$-gonal $2$-trees.
\newblock {\em Journal of Combinatorial Theory, Series A}, 106(2):193--219,
  2004.

\bibitem{palmer}
Edgar Palmer.
\newblock On the number of labeled $2$-trees.
\newblock {\em Journal of Combinatorial Theory}, 6(2):206--207, 1969.

\bibitem{sage}
W.\thinspace{}A. Stein et~al.
\newblock {\em Sage Mathematics Software (version 5.12)}.
\newblock The Sage Development Team, 2013.

\bibitem{vcuts}
Gareth~R. Wilkes.
\newblock On the structure of vertex cuts separating the ends of a graph,
  August 2013.

\end{thebibliography}
\bibliographystyle{plain}

\end{document}